\documentclass[a4paper]{article}%
\usepackage{amsmath}
\usepackage{amsfonts}
\usepackage{amssymb}
\usepackage{graphicx}%
\providecommand{\U}[1]{\protect\rule{.1in}{.1in}}
\newcommand{\remove}[1]{ }

\providecommand{\norm}[1]{\lVert#1\rVert}
\newtheorem{theorem}{Theorem}

\newtheorem{definition}[theorem]{Definition}

\newtheorem{proposition}[theorem]{Proposition}

\newenvironment{proof}[1][Proof]{\noindent\textbf{#1.} }{\ \rule{0.5em}{0.5em}}

\begin{document}

\title{Computability, Noncomputability, and Hyperbolic Systems}
\author{Ning Zhong\\DMS, University of Cincinnati,\\Cincinnati, OH 45221-0025, U.S.A.
\and Daniel S. Gra\c{c}a\\DM/FCT, Universidade do Algarve, C. Gambelas,\\8005-139 Faro, Portugal\\\& SQIG/Instituto de Telecomunica\c{c}\~{o}es, Lisbon, Portugal
\and Jorge Buescu\\DM/FCUL, University of Lisbon, Portugal \\\& CMAF, Lisbon, Portugal}
\maketitle

\begin{abstract}
In this paper we study the computability of the stable and unstable manifolds
of a hyperbolic equilibrium point. These manifolds are the essential feature
which characterizes a hyperbolic system. We show that (i) locally these
manifolds can be computed, but (ii) globally they cannot (though we prove they
are semi-computable). We also show that Smale's horseshoe, the first example
of a hyperbolic invariant set which is neither an equilibrium point nor a
periodic orbit, is computable.
\end{abstract}

\section{Introduction}

Dynamical systems are powerful objects which can be found in numerous
applications. However, this versatility does not come without cost: dynamical
systems are objects which are inherently very hard to study.

Recently, digital computers have been successfully used to analyze dynamical
systems, through the use of numerical simulations. However, with the known
existence of phenomena like the \textquotedblleft butterfly
effect\textquotedblright\ -- a small perturbation on initial conditions can be
exponentially amplified along time -- and the fact that numerical simulations
always involve some truncation errors, the reliability of such simulations for
providing information about the long-term evolution of the systems is questionable.

For instance, numerical simulations suggested the existence of a
\textquotedblleft strange\textquotedblright\ attractor for the Lorenz system
\cite{Lor63}, and it was widely believed that such an attractor existed.
However, the formal proof of its existence remained elusive, being at the
heart of the 14th from the list of 18 problems that the Fields medalist
S.\ Smale presented for the new millennium \cite{Sma98}. Finally, after a 35
years hiatus, a computer-aided formal proof was achieved in \cite{Tuc98},
\cite{Tuc99}, but this example painstakingly illustrates the difference
between numerical evidence and a full formal proof, even if computer-based.

For the above reasons, it is important to understand which properties can be
accurately computed with a computer, and those which cannot. We have
previously dwelled on this subject on our paper \cite{GZ10} which focused on
\textquotedblleft stable\textquotedblright\ dynamical systems having
hyperbolic attractors. The latter systems were extensively studied in the XXth
century and were thought to correspond to the class of \textquotedblleft
meaningful\textquotedblright\ dynamical systems. This happened for several
reasons: the \textquotedblleft stability\textquotedblright\ -- structural
stability -- which implies robustness of behavior to small perturbations was
believed to be of uttermost importance, since it was thought that only such
systems could exist in nature; there were also results (Peixoto's Theorem
\cite{Pei62}) which showed that, in the plane, these systems are dense and
their invariant sets (which include all attractors) can be fully characterized
(they can only be points or periodic orbits).There was hope that such results
would generalize for spaces of higher dimensions, but it was shattered by S.
Smale, who showed that there exist \textquotedblleft chaotic\textquotedblright%
\ hyperbolic invariant sets which are neither a point nor a periodic orbit
(Smale's \ horseshoe \cite{Sma67}) and that for dimensions $\geq3$
structurally stable systems are not dense \cite{Sma66}.

Moreover, the notion of structural stability was shown to be too strong a
requirement. In particular, the Lorenz attractor, which is embedded in a
system modeling weather evolution, is not structurally stable, although it is
stable to perturbations in the parameters defining the system \cite{Via00} (in
other words, robustness may not be needed for \emph{all} mathematical
properties of the system).

Despite the fact that systems with hyperbolic attractors are no longer
regarded as \textquotedblleft the\textquotedblright\ meaningful class of
dynamical systems for spaces of dimension $n\geq3$, they still play a central
role in the study of dynamical systems. In essence, what characterizes a
hyperbolic set is the existence at each point of an invariant splitting of the
tangent space into stable and unstable directions, which generate the local
stable and unstable manifolds (see Section \ref{Sec:Dynamical_systems}).

In this paper we will study the computability of the stable and
unstable manifolds for hyperbolic equilibrium points. The stable and
unstable manifolds are constructs which derive from the stable
manifold theorem. This theorem states that for each hyperbolic
equilibrium point $x_{0}$ (see Section \ref{Sec:Dynamical_systems}
for a definition), there exists a manifold $S$ such that a
trajectory on $S$ will converge to $x_{0}$ at an exponential rate as
$t\rightarrow+\infty$. Moreover every trajectory converging to
$x_{0}$ lies entirely in $S$. $S$ is call a stable manifold.
Similarly one can obtain the unstable manifold $U$, using similar
conditions when $t\rightarrow-\infty$. (The stable and unstable
manifolds of a hyperbolic equilibrium point are depicted in Fig.\
\ref{fig:hyperbolic}.) It is important to note that classical proofs
on existence of $S$ and $U$ are non-constructive. Thus computability (or, for the matter, semi-computability) of the stable/unstable manifold does not follow straightforwardly from these proofs.

We address the following basic question: given a dynamical system
and some hyperbolic equilibrium point, can we compute its stable and
unstable manifolds? We will show that the answer is twofold: (i) we
can compute the stable and unstable manifolds $S$ and $U$ given by
the stable manifold theorem which are, in some sense, local, since
there are (in general) trajectories starting in points which do not
lie in $S$ that will converge to $x_{0}$ (the set of all these point
united with $S$ would yield the global stable manifold -- see
Section \ref{Sec:Dynamical_systems}), but (ii) the global stable and
unstable manifolds are not, in general, computable from the
description of the system and the hyperbolic point. Since classical
proofs on existence of $S$ and $U$ are non-constructive, a different
approach is needed if one wishes to construct an algorithm that
computes $S$ and $U$. Our approach aiming at a constructive proof
makes use of function-theoretical treatment of the resolvents (see
the first paragraph of Section 4 for more details).

A most likely interpretation for these results is that, since the definition of
hyperbolicity is local, computability also applies locally. However global
homoclinic tangles can occur \cite{GH83}, leading to chaotic behavior for such
systems. So it should not be expected that the global behavior of stable and
unstable manifolds be globally computable in general, as indeed our results show.

In the end of the paper we also show that the prototypical example of an
hyperbolic invariant set which is neither a fixed point nor a periodic orbit
-- the Smale horseshoe -- is computable.

The computability of simple attractors -- hyperbolic periodic orbits and
equilibrium points -- was studied in our previous papers \cite{Zho09},
\cite{GZ10} for planar dynamics, as well as the computability of their
respective domains of attraction.

A number of papers study dynamical systems, while not exactly in the context
used here. For example the papers\ \cite{Moo90}, \cite{BBKT01}, \cite{Col05},
\cite{BY06}, \cite{Hoy07} provide interesting results about the long-term
behavior of dynamical systems using, among others, symbolic dynamics or a
statistical approach. See \cite{BGZ10} for a more detailed review of the
literature.%
%TCIMACRO{\TeXButton{Fig.: dynamics}{\begin{center}
%\begin{figure}[ptb]
%\begin{center}
%\includegraphics[width=5cm]{stable_unstable}
%\end{center}
%\caption{Dynamics near a hyperbolic equilibrium point.}
%\label{fig:hyperbolic}
%\end{figure}
%\end{center}}}%
%BeginExpansion
\begin{center}
\begin{figure}[ptb]
\begin{center}
\includegraphics[width=5cm]{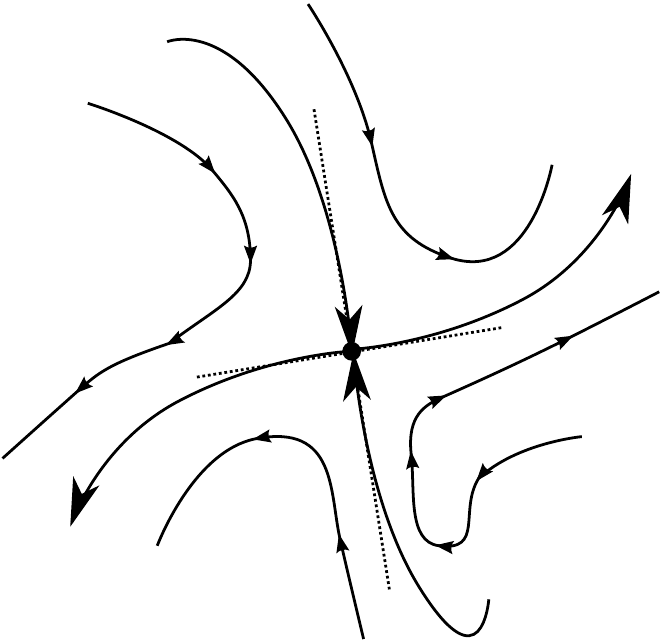}
\end{center}
\caption{Dynamics near a hyperbolic equilibrium point.}
\label{fig:hyperbolic}
\end{figure}
\end{center}%
%EndExpansion

\section{Dynamical systems\label{Sec:Dynamical_systems}}

In short, a dynamical system is a pair consisting of a state space where the
action occurs and a function $f$ which defines the evolution of the system
along time. See \cite{HS74} for a precise definition. In general one can
consider two kinds of dynamical systems: discrete ones, where time is discrete
and one obtains the evolution of the system by iterating the map $f$, and
continuous ones, where the evolution of the system along time is governed by a
differential equation of the type%
\begin{equation}
\dot{x}=f(x) \label{Eq:ODE}%
\end{equation}
where $t$ is the independent variable (the \textquotedblleft
time\textquotedblright) and $\dot{x}$ denotes the derivative $dx(t)/dt$.
Continuous-time systems can be translated to discrete-time using time-one maps
and, in some cases, the Poincar\'{e} map, and vice-versa (using the suspension
method). \ Therefore to study dynamical systems one can focus on
continuous-time ones.

Although the essential feature of hyperbolic attractors is the existence of an
invariant splitting of the tangent space into stable and unstable directions
generating the local stable and unstable manifolds, for a hyperbolic
equilibrium point $x_{0}$, it can be described equivalently in terms of the
linearization of the flow around $x_{0}.$ We recall that $x_{0}$ is an
equilibrium point of (\ref{Eq:ODE}) iff $f(x_{0})=0$.

\begin{definition}
An equilibrium point $x_{0}$ of (\ref{Eq:ODE}) is hyperbolic if none of the
eigenvalues of $Df(x_{0})$ has zero real part.
\end{definition}

According to the value of the real part of these eigenvalues, one can
determine the behavior of the linearized flow near $x_{0}$: if the eigenvalue
has negative real part, then the flow will converge to $x_{0}$ when it follows
the direction given by the eigenvector associated to this eigenvalue. A
similar behavior will happen when the eigenvalue has positive real part, with
the difference that convergence happens when $t\rightarrow-\infty$. See Fig.
\ref{fig:hyperbolic}. The space generated by the eigenvectors associated to
eigenvalues of $Df(x_{0})$ with negative real part is called the stable
subspace $E^{s}_{Df(x_{0})}$, and the space generated by the eigenvectors
associated to eigenvalues with positive real part is called the unstable
subspace $E^{u}_{Df(x_{0})}$.

We now state the Stable Manifold Theorem (as it appears in \cite{Per01}).

\begin{proposition}
[Stable Manifold Theorem]Let $E$ be an open subset of $\mathbb{R}^{n}$
containing the origin, let $f\in C^{1}(E)$, and let $\phi_{t}$ be the flow of
the system (\ref{Eq:ODE}). Suppose that $f(0)=0$ and that $Df(0)$ has $k$
eigenvalues with negative real part and $n-k$ eigenvalues with positive real
part (\emph{i.e.} 0 is a hyperbolic equilibrium point). Then there exists a
$k$-dimensional differentiable manifold $S$ tangent to the stable subspace
$E^{s}_{Df(0)}$ such that for all $t\geq0$, $\phi_{t}(S)\subseteq S$ and for
all $x_{0}\in S$%
\[
\lim_{t\rightarrow+\infty}\phi_{t}(x_{0})=0;
\]
and there exists an $n-k$ dimensional differentiable manifold $U$ tangent to
the unstable subspace $E^{u}_{Df(0)}$ such that for all $t\leq0$, $\phi
_{t}(S)\subseteq S$ and for all $x_{0}\in U$%
\[
\lim_{t\rightarrow-\infty}\phi_{t}(x_{0})=0.
\]

\end{proposition}

From the local stable and unstable manifolds given by the Stable Manifold
Theorem, one can define the \emph{global stable and unstable manifolds} of
(\ref{Eq:ODE}) at a hyperbolic equilibrium point $x_{0}$ by%
\begin{align}
W_{f}^{s}(x_{0}) &  = \bigcup_{j=0}^{\infty}
\phi_{-j}(S)\label{Eq:global}\\
W_{f}^{u}(x_{0}) &  =
\bigcup_{j=0}^{\infty}
\phi_{j}(U),\nonumber
\end{align}
respectively, where
\[
\phi_{t}(A)=\{x(t)|x\text{ is a solution of (\ref{Eq:ODE}) with }x(0)\in A\}.
\]
We note that $W^{s}(x_{0})$ and $W^{u}(x_{0})$ are $F_{\sigma}$-sets of
$\mathbb{R}^{n}$ (recall that a subset $F\subseteq\mathbb{R}^{n}$ is called an
$F_{\sigma}$-set if $F=\bigcup_{j=0}^{\infty}A_{j}$, where $A_{j}$,
$j\in\mathbb{N}$, is a closed subset of $\mathbb{R}^{n}$).

We end this section by introducing Smale's horseshoe \cite{Sma67}. We refer
the reader to \cite{GH83} for a more thorough discussion of this set. In
essence, Smale's horseshoe appears when we consider a map $f$ defined over
$S=[0,1]^{2}$. This map is bijective and performs a linear vertical expansion
of $S$, and a linear horizontal contraction of $S$, by factors $\mu>1$ and
$0<\lambda<1$, respectively, followed by a folding.

\begin{definition}
In the conditions defined above (see \cite{GH83} or \cite{HSD04} for precise
definitions), the Smale horseshoe is the set $\Lambda$ given by%
\[
\Lambda=%
%TCIMACRO{\dbigcap \limits_{j=-\infty}^{+\infty}}%
%BeginExpansion
{\displaystyle\bigcap\limits_{j=-\infty}^{+\infty}}
%EndExpansion
f^{j}(S).
\]
This set is invariant for the function $f$ (i.e.\ $f(\Lambda)=\Lambda$).
\end{definition}

\section{Computable analysis}

Now that we have introduced the main concepts of dynamical systems theory we
will use, we need to introduce concepts related to computability. The theory
of computation can be rooted in the seminal work of Turing, Church, and
others, which provided a framework in which to achieve computation over
discrete identities or, equivalently, over the integers.

However, this definition was not enough to cover computability over continuous
structures, and was then developed by other authors such as Turing himself
\cite{Tur36}, Grzegorczyk \cite{Grz57}, or Lacombe \cite{Lac55} to originate
\emph{computable analysis}.

The idea underlying computable analysis to compute over a set $A$ is to encode
each element $a$ of $A$ by a countable sequence of symbols from a finite
alphabet (called a $\rho$-\emph{name} for $a$). Each sequence can encode at
most one element of $A$. The more elements we have from a sequence encoding
$a$, the more precisely we can pinpoint $a$. From this point of view, it
suffices to work only with names when performing a computation over $A$. To
compute with names, we use Type-2 machines, which are similar to Turing
machines, but (i) have a read-only tape, where the input (i.e.\ the sequence
encoding it) is written; (ii) have a write-only output tape, where the head
cannot move back and the sequence encoding the output is written. For more
details the reader is referred to \cite{PR89}, \cite{Ko91}, \cite{Wei00}.

At any finite amount of time we can halt the computation, and we will have a
partial result on the output tape. The more time we wait, the more accurate
this result will be. We now introduce notions of computability over
$\mathbb{R}^{n}$.

\begin{definition}
\begin{enumerate}
\item A sequence $\{r_{k}\}$ of rational numbers is called a $\rho$-name of a
real number $x$ if there are three functions $a,b$ and $c$ from $\mathbb{N}$
to $\mathbb{N}$ such that for all $k\in\mathbb{N}$, $r_{k}=(-1)^{a(k)}%
\frac{b(k)}{c(k)+1}$ and
\begin{equation}
\left\vert r_{k}-x\right\vert \leq\frac{1}{2^{k}}. \label{PourElReal}%
\end{equation}

\item A double sequence $\{r_{l,k}\}_{l,k\in\mathbb{N}}$ of rational numbers
is called a $\rho$-name for a sequence $\{x_{l}\}_{l\in\mathbb{N}}$ of real
numbers if there are three functions $a,b,c$ from $\mathbb{N}^{2}$ to
$\mathbb{N}$ such that, for all $k, l\in\mathbb{N}$, $r_{l,k}=(-1)^{a(l,k)}%
\frac{b(l,k)}{c(l,k)+1}$ and
\[
\left\vert r_{l,k}-x_{l}\right\vert \leq\frac{1}{2^{k}}.
\]

\item A real number $x$ (a sequence $\{x_{l}\}_{l\in\mathbb{N}}$ of real
numbers) is called computable if it has a computable $\rho$-name, i.e. the
functions $a$, $b$, and $c$ are computable or, equivalently, there is a Type-2
machine that computes a $\rho$-name of $x$ ($\{x_{l}\}_{l\in\mathbb{N}}$,
respectively) without any input.
\end{enumerate}
\end{definition}

In general, $\rho$-names for real numbers do not need to be exactly those
described above: different definitions can provide the same set of computable
points (e.g.\ the $\rho$-name could be a sequence of intervals containing
$x_{0}$ with diameter $1/k$ for $k\geq1$, etc.). A notable exception is the
decimal expansion of $x_{0}$, which cannot be used since it can lead to
undesirable behavior \cite{Tur37} because this representation does not respect
the topology of the space $\mathbb{R}$.

The notion of the $\rho$-name can be extended to points in $\mathbb{R}^{n}$ as
follows: a sequence $\{(r_{1k},r_{2k},\ldots,r_{nk})\}_{k\in\mathbb{N}}$ of
rational vectors is called a $\rho$-name of $x=(x_{1},x_{2},\ldots,x_{n}%
)\in\mathbb{R}^{n}$ if $\{r_{jk}\}_{k\in\mathbb{N}}$ is a $\rho$-name of
$x_{j}$, $1\leq j\leq n$. Using $\rho$-names, we can define computable functions.

\begin{definition}
Let $X$ and $Y$ be two sets, where $\rho$-names can be defined for elements of
$X$ and $Y$. A function $f:X\rightarrow Y$ is computable if there is a Type-2
machine such that on any $\rho$-name of $x\in X$, the machine computes as
output a $\rho$-name of $f(x)\in Y$.
\end{definition}

Next we present a notion of computability for open and closed subsets of
$\mathbb{R}^{n}$ (cf. \cite{Wei00}, Definition 5.1.15). We implicitly use
$\rho$-names. For instance, to obtain names of open subsets of $\mathbb{R}%
^{n}$, we note that the set of rational balls $B(a,r)=\{x\in\mathbb{R}%
^{n}:\left\vert x-a\right\vert <r\}$, where $a\in\mathbb{Q}^{n}$ and
$r\in\mathbb{Q}$, is a basis for the standard topology over $\mathbb{R}^{n}$.
Thus a sequence $\{(a_{k},r_{k})\}_{k\in\mathbb{N}}$ such that $E=\cup
_{k=0}^{\infty}B(a_{k},r_{k})$ gives a $\rho$-name for the open
set $E$.

\begin{definition}
\label{def_r_e_open}

\begin{enumerate}
\item An open set $E\subseteq\mathbb{R}^{n}$ is called recursively enumerable
(r.e.\ for short) open if there are computable sequences $\{a_{k}\}$ and
$\{r_{k}\}$, $a_{k}\in\mathbb{Q}^{n}$ and $r_{k}\in\mathbb{Q}$, such that
\[
E=\cup_{k=0}^{\infty}B(a_{k},r_{k}).
\]
Without loss of generality one can also assume that for any $k\in\mathbb{N}$,
the closure of $B(a_{k},r_{k})$, denoted as $\overline{B(a_{k},r_{k})}$, is
contained in $E$.

\item A closed subset $A\subseteq\mathbb{R}^{n}$ is called r.e.\ closed if
there exists a computable sequence $\{b_{k}\}$, $b_{k}\in\mathbb{Q}^{n}$, such
that $\{b_{k}\}$ is dense in $A$. $A$ is called co-r.e. \ closed if its
complement $A^{c}$ is r.e. open. $A$ is called computable (or recursive) if it
is both r.e.\ and co-r.e.

\item An open set $E\subseteq\mathbb{R}^{n}$ is called computable (or
recursive) if $E$ is r.e.\ open and its complement $E^{c}$ is r.e.\ closed.

\item A compact set $K\subseteq\mathbb{R}^{n}$ is called computable if it is
computable as a closed set and, in addition, there is a rational number $b$
such that $||x||\leq b$ for all $x\in K$.
\end{enumerate}
\end{definition}

In the rest of the paper, we will work exclusively with $C^{1}$ functions $f:
E\to\mathbb{R}^{n}$, where $E$ is an open subset of $\mathbb{R}^{n}$. Thus it
is desirable to present an explicit $\rho$-name for such functions.

\begin{definition}
\label{Def:rho_C1_function} Let $E=\bigcup_{k=0}^{\infty}B(a_{k},r_{k})$,
$a_{k}\in\mathbb{Q}^{n}$ and $r_{k}\in\mathbb{Q}$, be an open subset of
$\mathbb{R}^{n}$ (assuming that the closure of $B(a_{k},r_{k})$ is contained
in $E$) and let $f:E\rightarrow\mathbb{R}^{n}$ be a continuously
differentiable function. Then a ($C^{1}$) $\rho$-name of $f$ is a sequence
$\{P_{l}\}$ of polynomials ($P_{l}:\mathbb{R}^{n}\rightarrow\mathbb{R}^{n}$)
with rational coefficients such that
\[
d_{C^{1}(E)}(f,P_{l})\leq2^{-l}\quad\mbox{for all $l\in \mathbb{N}$}
\]
where
\[
d_{C^{1}(E)}(f,P_{l})=\sum_{k=0}^{\infty}2^{-k}\left(  \frac{||f-P_{l}||_{k}%
}{1+||f-P_{l}||_{k}}+\frac{||Df-DP_{l}||_{k}}{1+||Df-DP_{l}||_{k}}\right)
\]
and
\[
||g||_{k}=\max_{x\in\overline{B(a_{k},r_{k})}}|g(x)|.
\]

\end{definition}

We observe that this $\rho$-name of $f$ contains information on both $f$ and
$Df$ in the sense that $(P_{1}, P_{2}, \ldots)$ is a $\rho$-name of $f$ while
$(DP_{1}, DP_{2}, \ldots)$ is a $\rho$-name of $Df$. See \cite{ZW03} for
further details.

Throughout the rest of this paper, unless otherwise mentioned, we will assume
that, in (\ref{Eq:ODE}), $f$ is continuously differentiable on an open subset
of $\mathbb{R}^{n}$ and we will use the above $\rho$-name for $f$.

\section{Computable stable manifold theorem}

\label{sec_computable}

The stable manifold theorem states that near a hyperbolic equilibrium point
$x_{0}$, the nonlinear system
\begin{equation}
\dot{x}=f(x(t)) \label{e1}%
\end{equation}
has stable and unstable manifolds $S$ and $U$ tangent to the stable and
unstable subspaces ${\mathbb{E}}_{A}^{s}$ and ${\mathbb{E}}_{A}^{u}$ of the
linear system
\begin{equation}
\dot{x}=Ax
\end{equation}
where $A=Df(x_{0})$ is the gradient matrix of $f$ at $x_{0}$. The classical
proof of the theorem relies on the Jordan canonical form of $A$. To reduce $A$
to its Jordan form, one needs to find a basis of generalized eigenvectors.
Since the process of finding eigenvectors from corresponding eigenvalues is
not continuous in general, it is a non-computable process. Thus if one wishes
to construct an algorithm that computes some $S$ and $U$ of (\ref{e1}) at
$x_{0}$, a different method is needed. We will make use of an analytic, rather
than algebraic, approach to the eigenvalue problem that allows us to compute
$S$ and $U$ without calling for eigenvectors. The analytic approach is based
on function-theoretical treatment of the resolvents (see, \emph{e.g.},
\cite{Nag42}, \cite{Kat49}, \cite{Kat50}, and \cite{Rob95}). \newline

Let us first show that the stable and unstable subspaces are computable from
$A$ for the linear hyperbolic systems $\dot{x}=Ax$. We begin with several
definitions. Let $\mathfrak{A}_{H}$ denote the set of all $n\times n$ matrices
such that the linear differential equation $\dot{x}=Ax$, $x\in{\mathbb{R}}%
^{n}$, is hyperbolic, where a linear differential equation $\dot{x}=Ax$ is
defined to be hyperbolic if all the eigenvalues of $A$ have nonzero real part.
The Hilbert-Schmidt norm is used for $A\in\mathfrak{A}_{H}$: $\norm{A}=\sqrt
{\sum_{i=1}^{n}\sum_{j=1}^{n}|a_{ij}|^{2}}$, where $a_{ij}$ is the
$ij^{\mbox{th}}$ entry of $A$. For each $A\in\mathfrak{A}_{H}$, define the
stable subspace $\mathbb{E}^{s}_{A}$ and unstable subspace $\mathbb{E}^{u}%
_{A}$ to be
\begin{align*}
{\mathbb{E}}^{s}_{A}  &  = \mbox{span}\left\{  v\in{\mathbb{R}}^{n}:
\mbox{$v$ is a generalized
eigenvector for an eigenvalue $\lambda$ of $A$} \right. \\
&  \qquad\quad\left.  \mbox{with $Re(\lambda)<0$}\right\}
\end{align*}
\begin{align*}
{\mathbb{E}}^{u}_{A}  &  = \mbox{span}\left\{  v\in{\mathbb{R}}^{n}:
\mbox{$v$ is a generalized
eigenvector for an eigenvalue $\lambda$ of $A$} \right. \\
&  \qquad\quad\left.  \mbox{with $Re(\lambda)>0$}\right\}
\end{align*}
Then ${\mathbb{R}}^{n}={\mathbb{E}}^{s}_{A}\bigoplus{\mathbb{E}}^{u}_{A}$. The
stable subspace $\mathbb{E}^{s}_{A}$ is the set of all vectors which contract
exponentially forward in time while the unstable subspace $\mathbb{E}^{u}_{A}$
is the set of all vectors which contract backward in time.

As mentioned above, the process of finding eigenvectors from corresponding
eigenvalues is not computable; thus the algebraic approach to ${\mathbb{E}%
}^{s}_{A}$ and ${\mathbb{E}}^{u}_{A}$ is a non-computable process. Of course,
this doesn't necessarily imply that it is impossible to compute ${\mathbb{E}%
}^{s}_{A}$ and ${\mathbb{E}}^{u}_{A}$ from $A$, but rather this particular
classical approach fails to be computable. So even for the linear hyperbolic
system $\dot{x}=Ax$, one needs a different approach to treat the
stable/unstable subspace when computability concerned. The approach used below
to treat the stable/unstable subspace is analytic, rather than algebraic. Let
$p_{A}(\lambda)$ be the characteristic polynomial for $A$, $\gamma_{1}$ be any
closed curve in the left half of the complex plane that surrounds (in its
interior) all eigenvalues of $A$ with negative real part and is oriented
counterclockwise, and $\gamma_{2}$ any closed curve in the right half of the
complex plane that surrounds all eigenvalues of $A$ with positive real part,
again with counterclockwise orientation. Then
\[
P^{1}_{A}{\mathbb{R}}^{n}={\mathbb{E}}^{s}_{A}, \qquad P^{2}_{A}{\mathbb{R}%
}^{n}={\mathbb{E}}^{u}_{A}%
\]
where
\[
P^{1}_{A}v=\frac{1}{2\pi i}\int_{\gamma_{1}}(\xi I-A)^{-1}vd\xi, \quad
P^{2}_{A}v=\frac{1}{2\pi i}\int_{\gamma_{2}}(\xi I-A)^{-1}vd\xi
\]
(See Section 4.6 of \cite{Rob95}.)

\begin{theorem}
\label{t0} The map $H^{s}:{\mathfrak{A}}_{H}\rightarrow\mathcal{A}%
({\mathbb{R}}^{n}%
)=\{X|\mbox{$X\subseteq\mathbb{R}^n$ is a closed subset of $\mathbb{R}^n$}\}$,
$A\mapsto{\mathbb{E}}_{A}^{s}$, is computable, where ${\mathfrak{A}}_{H}$ is
represented \textquotedblleft entrywise": $\rho=(\rho_{ij})_{i,j=1}^{n}$ is a
name of $A$ if $\rho_{ij}$ is a $\rho$-name of $a_{ij}$, the $ij$th entry of
$A$.
\end{theorem}

\begin{proof}
First we observe that the map $A\mapsto$ the eigenvalues of $A$,
$A\in{\mathfrak{A}}_{H}$, is computable. Assume that $\lambda_{1}%
,\ldots,\lambda_{k},\mu_{k+1},\ldots,\mu_{n}$ (counting multiplicity) are
eigenvalues of $A$ with $Re(\lambda_{j})<0$ for $1\leq j\leq k$ and
$Re(\mu_{j})>0$ for $k+1\leq j\leq n$. Then from the $\rho$-names of the
eigenvalues, one can compute two rectangular closed curves $\gamma_{A}^{1}$
and $\gamma_{A}^{2}$, where $\gamma_{A}^{1}$ is in the left half of the
complex plane that surrounds all eigenvalues $\lambda_{j}$ with $1\leq j\leq
k$ and $\gamma_{A}^{2}$ is in the right half of the complex plane that
surrounds all eigenvalues $\mu_{j}$ with $k+1\leq j\leq n$, and both
$\gamma_{A}^{1}$ and $\gamma_{A}^{2}$ are oriented counterclockwise. From
$\gamma_{A}^{1}$ and $\gamma_{2}^{A}$ one can further computes the maps
$P_{A}^{1},P_{A}^{2}:{\mathbb{R}}^{n}\rightarrow{\mathbb{R}}^{n}$ , where
$P_{A}^{1}v=\frac{1}{2\pi i}\int_{\gamma_{A}^{1}}(\xi I-A)^{-1}vd\xi$ and
$P_{A}^{2}v=\frac{1}{2\pi i}\int_{\gamma_{A}^{2}}(\xi I-A)^{-1}vd\xi$. We note
that, on the one hand, $\mathbb{E}^{s}_{A}=(P^{2}_{A})^{-1}(\{0\})$, and on
the other hand, $\mathbb{E}^{s}_{A}=\overline{\mathbb{E}^{s}_{A}}%
=\overline{P^{1}_{A}\mathbb{R}^{n}}$, where $\overline{K}$ denotes the closure
of the set $K$. Then by Theorem 6.2.4 of \cite{Wei00}, $\mathbb{E}^{s}_{A}$ is
both r.e. and co-r.e. closed, thus computable. The same argument shows that
${\mathbb{E}}_{A}^{u}$ is computable from $A$.
\end{proof}

Next we present an effective version of the stable manifold theorem. Consider
the nonlinear system
\begin{equation}
\label{e0}\dot{x}=f(x(t))
\end{equation}
Assume that (\ref{e0}) defines a dynamical system, that is, the solution $x(t,
x_{0})$ to (\ref{e0}) with the initial condition $x(0)=x_{0}$ is defined for
all $t\in\mathbb{R}$. Since the system (\ref{e0}) is autonomous, if $p$ is a
hyperbolic equilibrium point, without loss of generality, we may assume that
$p$ is the origin $0$. \newline

\begin{theorem}
\label{t1} Let $f:\mathbb{R}^{n}\rightarrow\mathbb{R}^{n}$ be a $C^{1}%
$-computable function (meaning that both $f$ and $Df$ are computable). Assume
that the origin $0$ is a hyperbolic equilibrium point of (\ref{e0}) such that
$Df(0)$ has $k$ eigenvalues with negative real part and $n-k$ eigenvalues with
positive real part (counting multiplicity), $0<k\leq n$. Let $x(t,x_{0})$
denote the solution of (\ref{e0}) with the initial value $x_{0}$ at $t=0$.
Then there is a (Turing) algorithm that computes a $k$-dimensional manifold
$S\subset{\mathbb{R}}^{n}$ containing $0$ such that

\begin{itemize}
\item[(i)] For all $x_{0}\in S$, $\lim_{t\rightarrow+\infty}x(t,x_{0})=0$;

\item[(ii)] There are three positive rational numbers $\gamma$, $\epsilon$,
and $\delta$ such that $|x(t,x_{0})|\leq\gamma2^{-\epsilon t}$ for all
$t\geq0$ whenever $x_{0}\in S$ and $|x_{0}|\leq\delta$.
\end{itemize}

Moreover, if $k<n$, then a rational number $\eta$ and a ball $D$ with center
at the origin can be computed from $f$ such that for any solution $x(t,x_{0})$
to the equation (\ref{e0}) with $x_{0}\in D\setminus S$, $\{ x(t, x_{0}):
\ t\geq0\}\not \subset B(0, \eta)$ no matter how close the initial value
$x_{0}$ is to the origin.
\end{theorem}

\begin{proof}
First we note that under the assumption that $f$ is $C^{1}$-computable, the
solution map $x:\mathbb{R}\times\mathbb{R}^{n}\rightarrow\mathbb{R}^{n}$,
$(t,a)\mapsto x(t,a)$, is computable (\cite{GZB07}).

Let $f(x)=Ax+F(x)$, where $A=Df(0)$ and $F(x)=f(x)-Ax$. Then the equation
(\ref{e0}) can be written in the form of
\begin{equation}
\dot{x}=Ax+F(x) \label{e2}%
\end{equation}
The first step in constructing the desired algorithm is to break the flow
$e^{At}$ governed by the linear equation $\dot{x}=Ax$ into the stable and
unstable components, denoted as $I_{\Gamma_{1}}(t)$ and $I_{\Gamma_{2}}(t)$
respectively. By making use of an integral formula, we are able to show that
the breaking process is computable from $A$. The details for the first step:
Since $F(x)=f(x)-Ax$, it follows that $F(0)=0$, $DF(0)=0$, $F$ and $DF$ are
both computable because $f$ and $Df$ are computable functions by assumption.
Thus there is a computable modulus of continuity $d:\mathbb{N}\rightarrow
\mathbb{N}$ such that
\begin{equation}
|F(x)-F(y)|\leq2^{-m}|x-y|\ \ \mbox{whenever}\ \ |x|\leq2^{-d(m)}%
\ \&\ |y|\leq2^{-d(m)} \label{ine-0}%
\end{equation}
Since $Df$ is computable, all entries in the matrix $A$ are computable;
consequently, the coefficients of the characteristic polynomial
$\mbox{det}(A-\lambda I_{n})$ of $A$ are computable, where
$\mbox{det}(A-\lambda I_{n})$ denotes the determinant of $A-\lambda I_{n}$ and
$I_{n}$ is the $n\times n$ unit matrix. Thus all eigenvalues of $A$ are
computable, for they are zeros of the computable polynomial
$\mbox{det}(A-\lambda I_{n})$. Assume that $\lambda_{1},\ldots,\lambda_{k}%
,\mu_{k+1},\ldots,\mu_{n}$ (counting multiplicity) are eigenvalues of $A$ with
$Re(\lambda_{j})<0$ for $1\leq j\leq k$ and $Re(\mu_{j})>0$ for $k+1\leq j\leq
n$, where $Re(z)$ denotes the real part of the complex number $z$. Then two
rational numbers $\sigma>0$ and $\alpha>0$ can be computed from the
eigenvalues of $A$ such that $Re(\lambda_{j})<-(\alpha+\sigma)$ for $1\leq
j\leq k$ and $Re(\mu_{j})>\sigma$ for $k+1\leq j\leq n$. We break $\alpha$
into two parts for later use: Let $\alpha_{1}$ and $\alpha_{2}$ be two
rational numbers such that
\begin{equation}
0<\alpha_{1}<\alpha\text{ \ \ \ and \ \ \ }\alpha_{1}+\alpha_{2}=\alpha.
\label{Eq:alphas}%
\end{equation}

Let $M$ be a natural number such that $M>\max\{\alpha+\sigma,1\}$ and
$\max\{|\lambda_{j}|,|\mu_{l}|:1\leq j\leq k,k+1\leq l\leq n\}\leq M-1$. We
now construct two simple piecewise-smooth close curves $\Gamma_{1}$ and
$\Gamma_{2}$ in $\mathbb{R}^{2}$: $\Gamma_{1}$ is the boundary of the
rectangle with the vertices $(-\alpha-\sigma,M)$, $(-M,M)$, $(-M,-M)$, and
$(-\alpha-\sigma,-M)$, while $\Gamma_{2}$ is the boundary of the rectangle
with the vertices $(\sigma,M)$, $(M,M)$, $(M,-M)$, and $(\sigma,-M)$. Then
$\Gamma_{1}$ with positive direction (counterclockwise) encloses in its
interior all the $\lambda_{j}$ for $1\leq j\leq k$ and $\Gamma_{2}$ with
positive direction encloses all the $\mu_{j}$ for $k+1\leq j\leq n$ in its
interior. We observe that for any $\xi\in\Gamma_{1}\bigcup\Gamma_{2}$, the
matrix $A-\xi I_{n}$ is invertible. Since the function $g:\Gamma_{1}%
\bigcup\Gamma_{2}\rightarrow\mathbb{R}$, $g(\xi)=||(A-\xi I_{n})^{-1}||$, is
computable (see for example \cite{Zho09}), where $(A-\xi I_{n})^{-1}$ is the
inverse of the matrix $A-\xi I_{n}$, the maximum of $g$ on $\Gamma_{1}%
\bigcup\Gamma_{2}$ is computable. Let $K_{1}\in\mathbb{N}$ be an upper bound
of this computable maximum. Now for any $t\in\mathbb{R}$, from (5.47) of
\cite{Kat95},
\begin{align}
e^{At}  &  =-\frac{1}{2\pi i}\int_{\Gamma_{1}}e^{\xi t}(A-\xi I_{n})^{-1}%
d\xi-\frac{1}{2\pi i}\int_{\Gamma_{2}}e^{\xi t}(A-\xi I_{n})^{-1}%
d\xi\label{e-lin}\\
&  =I_{\Gamma_{1}}(t)+I_{\Gamma_{2}}(t)\nonumber
\end{align}
We recall that $e^{At}$ is the solution to the linear equation $\dot{x}=Ax$.
Since $A$ is computable and integration is a computable operator, it follows
that $I_{\Gamma_{1}}$ and $I_{\Gamma_{2}}$ are computable. A simple
calculation shows that $||-\frac{1}{2\pi i}\int_{\Gamma_{1}}e^{t\xi}(A-\xi
I_{n})^{-1}d\xi||\leq4K_{1}Me^{-(\alpha+\sigma)t}/\pi$ for $t\geq0$ and
$||-\frac{1}{2\pi i}\int_{\Gamma_{2}}e^{t\xi}(A-\xi I_{n})^{-1}d\xi
||\leq4K_{1}Me^{\sigma t}/\pi$ for $t\leq0$. Let
\begin{equation}
K=4MK_{1} \label{eK}%
\end{equation}
Then
\begin{equation}
\mbox{$||I_{\Gamma_1}(t)||\leq Ke^{-(\alpha +\sigma) t}$ for
$t\geq 0$ \ and \ $||I_{\Gamma_2}(t)||\leq Ke^{\sigma t}$ for
$t\leq 0$} \label{eI}%
\end{equation}
The two estimates in (\ref{eI}) show that $I_{\Gamma_{1}}(t)$ and
$I_{\Gamma_{2}}(t)$ are stable component and unstable component of $e^{A}$.
The first step is now complete.

The second step in the construction is to compute a ball, $B(0, r)$,
surrounding the hyperbolic equilibrium point 0 such that $B(0, r)$ contains a
set of potential solutions $x(t, x_{0})$ to (\ref{e2}) satisfying the
conditions (i) and (ii) described in Theorem \ref{t1}. The details for step 2:
\newline

\noindent\textbf{Claim 1.} Consider the integral equation
\begin{equation}
u(t,a)=I_{\Gamma_{1}}(t)a+\int_{0}^{t}I_{\Gamma_{1}}(t-s)F(u(s,a))ds-\int
_{t}^{\infty}I_{\Gamma_{2}}(t-s)F(u(s,a))ds \label{e3}%
\end{equation}
where the constant vector $a$ is a parameter. If $u(t,a)$ is a continuous
solution to the integral equation, then it satisfies the differential equation
(\ref{e2}) with the initial condition $u(0,a)=I_{\Gamma_{1}}(0)a-\int
_{0}^{\infty}I_{\Gamma_{2}}(-s)F(u(s,a))ds$. \newline

\noindent\textbf{Proof of Claim 1.} See Appendix 1. \newline

We observe that the solution to the integral equation is the fixed point of
the operator defined on the right hand side of the equation (\ref{e3}). Next
we compute an integer $m_{0}$ such that $2^{-m_{0}}\leq\frac{\sigma}{4K}$, and
then set
\begin{equation}
r=2^{-d(m_{0})}/2K,\qquad B=B(0,r)=\{a\in\mathbb{R}^{n}:|a|<r\} \label{er}%
\end{equation}
where the computable function $d$ is as in (\ref{ine-0}). \newline

\noindent\textbf{Claim 2.} For any $a\in B$ and $t\geq0$, define the
successive approximations as follows:
\begin{align}
\label{e4}u^{(0)}(t, a)  &  = 0\\
u^{(j)}(t, a)  &  = I_{\Gamma_{1}}(t)a+\int_{0}^{t}I_{\Gamma_{1}%
}(t-s)F(u^{(j-1)}(s, a))ds\nonumber\\
&  -\int_{t}^{\infty}I_{\Gamma_{2}}(t-s)F(u^{(j-1)}(s,a))ds, \ \ j\geq
1\nonumber
\end{align}
then the following three inequalities hold for all $j\in\mathbb{N}$:
\begin{equation}
\label{ine-1}|u^{(j)}(t,a)-u^{(j-1)}(t,a)|\leq K|a|e^{-\alpha_{1}t}/2^{j-1}%
\end{equation}
\begin{equation}
\label{ine-2}|u^{(j)}(t, a)|\leq2^{-d(m_{0})}e^{-\alpha_{1}t}%
\end{equation}
and for any $\tilde{a}\in B$,
\begin{equation}
\label{ine-3}|u^{(j)}(t, a)-u^{(j)}(t, \tilde{a})|\leq3K|a-\tilde{a}|\\
\end{equation}

\noindent\textbf{Proof of Claim 2.} See Appendix 2. \newline

It then follows from (\ref{e4}) and (\ref{ine-1}) that $\{u^{(j)}%
(t,a)\}_{j=1}^{\infty}$ is a computable Cauchy sequence effectively convergent
to the solution $u(t,a)$ of the integral equation (\ref{e3}), uniformly in
$t\geq0$ and $a\in B$. Consequently the solution, $t,a\mapsto u(t,a)$, is
computable. Furthermore, (\ref{ine-2}) and (\ref{ine-3}) imply that for all
$t\geq0$ and $a,\tilde{a}\in B$,
\begin{equation}
|u(t,a)|\leq2^{-d(m_{0})}e^{-\alpha_{1}t},\qquad|u(t,a)-u(t,\tilde
{a}|<3K|a-\tilde{a}| \label{ine-2-3-c}%
\end{equation}
Thus $\lim_{t\rightarrow\infty}u(t,a)=0\quad\mbox{for all $a\in
B$}$. Moreover, the first inequality in (\ref{ine-2-3-c}) shows that $u(t, a)$
satisfies condition (ii) of Theorem \ref{t1} with $\gamma=2^{-d(m_{0})}$,
$\epsilon=\alpha_{1}$, and $\delta=2^{-d(m_{0})}/2K$.\newline

Although Claim 2 shows that for any $a\in B$, the integral equation (\ref{e3})
has a computable solution $u(t, a)$ and this solution satisfies the conditions
(i) and (ii) of Theorem \ref{t1}, we may not take $B$ as a desired stable
manifold $S$ because if $u(0, a)\neq a$, then $u(t, a)$ is not a solution to
the equation (\ref{e0}) with the initial value $a$ at $t=0$. Nevertheless, the
set $B$ provides a pool of potential solutions to (\ref{e0}) on a stable
manifold. This leads us to the next step of the proof. \newline

The last step in the construction of the desired algorithm is to extract a
$k$-dimensional manifold $S$ from $B$, using a computable process, such that
for any $a\in S$, $u(0, a)=a$. Then if we set $x(t, a)=u(t,a)$ for $a\in S$,
by Claims 1 and 2, $x(t, a)$ is the solution to the differential equation
(\ref{e2}) with the initial value $a$ at $t=0$ and satisfies the conditions
(i) and (ii) of Theorem \ref{t1}. Now for the details.

Let $P_{1}=I_{\Gamma_{1}}(0)=-\frac{1}{2\pi i}\int_{\Gamma_{1}}(A-\xi
I_{n})^{-1}d\xi$ and $P_{2}=I_{\Gamma_{2}}(0)=-\frac{1}{2\pi i}\int
_{\Gamma_{2}}(A-\xi I_{n})^{-1}d\xi$. Then $P_{1}\mathbb{R}^{n}={\mathbb{E}%
}_{A}^{s}$, $P_{2}\mathbb{R}^{n}={\mathbb{E}}_{A}^{u}$, $\mathbb{R}%
^{n}={\mathbb{E}}_{A}^{s}\bigoplus{\mathbb{E}}_{A}^{u}$ with
$\mbox{dim}{\mathbb{E}}_{A}^{s}=k$ and $\mbox{dim}{\mathbb{E}}_{A}^{u}=n-k$,
$P_{j}P_{k}=\delta_{jk}P_{j}$ ($\delta_{jk}=1$ if $j=k$ and $\delta_{jk}=0$ if
$j\neq k$), and $P_{1}+P_{2}=I$ is the identity map on $\mathbb{R}^{n}$
(\emph{c.f.} $\S 1.5.3$ and $\S 1.5.4$, \cite{Kat95}; $\S 4.6$ of
\cite{Rob95}). Since $A$ is computable, so are $P_{1}$ and $P_{2}$. Moreover,
$I_{\Gamma_{1}}(t)P_{2}=0$ for any $t\in\mathbb{R}$ as the following
calculation shows: Let $R(\xi)$ denote $(A-\xi I_{n})^{-1}$. Then we have%
\begin{align*}
R(\xi_{1})-R(\xi_{2})  &  =R(\xi_{1})(A-\xi_{2}I_{n})R(\xi_{2})-R(\xi
_{1})(A-\xi_{1}I_{n})R(\xi_{2})\\
&  =R(\xi_{1})[(A-\xi_{2}I_{n})-(A-\xi_{1}I_{n})]R(\xi_{2})\\
&  =R(\xi_{1})(\xi_{1}-\xi_{2})I_{n}R(\xi_{2})\\
&  =(\xi_{1}-\xi_{2})R(\xi_{1})R(\xi_{2}).
\end{align*}
Using the last equation we obtain that for any $v\in\mathbb{R}^{n}$,
\begin{align}
I_{\Gamma_{1}}(t)P_{2}v  &  =\left(  \frac{1}{2\pi i}\right)  ^{2}\int
_{\Gamma_{1}}e^{\xi t}R(\xi)d\xi\int_{\Gamma_{2}}R(\xi^{\prime})vd\xi^{\prime
}\label{e5}\\
&  =\left(  \frac{1}{2\pi i}\right)  ^{2}\int_{\Gamma_{1}}\int_{\Gamma_{2}%
}e^{\xi t}R(\xi)R(\xi^{\prime})vd\xi d\xi^{\prime}\nonumber\\
&  =\left(  \frac{1}{2\pi i}\right)  ^{2}\int_{\Gamma_{1}}\int_{\Gamma_{2}%
}e^{\xi t}\frac{R(\xi)-R(\xi^{\prime})}{\xi-\xi^{\prime}}vd\xi d\xi^{\prime
}\nonumber\\
&  =\left(  \frac{1}{2\pi i}\right)  ^{2}\int_{\Gamma_{1}}e^{\xi t}\left(
\int_{\Gamma_{2}}\frac{R(\xi)}{\xi-\xi^{\prime}}d\xi^{\prime}-\int_{\Gamma
_{2}}\frac{R(\xi^{\prime})}{\xi-\xi^{\prime}}d\xi^{\prime}\right)
vd\xi\nonumber\\
&  =\left(  \frac{1}{2\pi i}\right)  ^{2}\int_{\Gamma_{1}}e^{\xi t}\left(
-\int_{\Gamma_{2}}\frac{R(\xi^{\prime})}{\xi-\xi^{\prime}}d\xi^{\prime
}\right)  vd\xi\nonumber\\
&  =\left(  \frac{1}{2\pi i}\right)  ^{2}\int_{\Gamma_{2}}R(\xi^{\prime
})\left(  \int_{\Gamma_{1}}\frac{e^{\xi t}}{\xi^{\prime}-\xi}d\xi\right)
vd\xi^{\prime}=0\nonumber
\end{align}
A similar computation shows that for any $t\in\mathbb{R}$,
\begin{equation}
I_{\Gamma_{2}}(t)P_{1}v=0 \quad\mbox{and} \quad P_{2}I_{\Gamma_{2}%
}(t)v=I_{\Gamma_{2}}(t)v, \quad v\in\mathbb{R}^{n} \label{e6}%
\end{equation}
Now let us use these results to compute the projections of $u(0,a)$ in
${\mathbb{E}}_{A}^{s}$ and ${\mathbb{E}}_{A}^{u}$: for any $a\in\mathbb{R}%
^{n}$,
\begin{align}
P_{1}u(0,a)  &  =P_{1}\left(  I_{\Gamma_{1}}(0)a-\int_{0}^{\infty}%
I_{\Gamma_{2}}(-s)F(u(s,a))ds\right) \nonumber\\
&  =P_{1}P_{1}a-\int_{0}^{\infty}P_{1}I_{\Gamma_{2}}(-s)F(u(s,a))ds\nonumber\\
&  =P_{1}a \label{Eq:P1_a}%
\end{align}
and
\begin{align*}
P_{2}u(0,a)  &  =P_{2}\left(  I_{\Gamma_{1}}(0)a-\int_{0}^{\infty}%
I_{\Gamma_{2}}(-s)F(u(s,a))ds\right) \\
&  =P_{2}P_{1}a-P_{2}\left(  \int_{0}^{\infty}I_{\Gamma_{2}}%
(-s)F(u(s,a))ds\right) \\
&  =-\int_{0}^{\infty}I_{\Gamma_{2}}(-s)F(u(s,a))ds
\end{align*}
We note that for any $a\in\mathbb{R}^{n}$,
\begin{align*}
I_{\Gamma_{1}}(t)a  &  =I_{\Gamma_{1}}(t)(P_{1}a+P_{2}a)\\
&  =I_{\Gamma_{1}}(t)P_{1}a+I_{\Gamma_{1}}(t)P_{2}a\\
&  =I_{\Gamma_{1}}(t)P_{1}a
\end{align*}
which implies that if the solution $u(t,a)$ of the integral equation
(\ref{e3}) is constructed by successive approximations (\ref{e4}), then
$P_{2}a$ does not enter the computation for $u(t,a)$ and thus may be taken as
zero. Therefore the projection of $u(0,a)$ in ${\mathbb{E}}_{A}^{u}$ satisfies
the equation
\begin{equation}
\label{e-P2}P_{2}u(0,a)=-\int_{0}^{\infty}I_{\Gamma_{2}}(-s)F(u(s,P_{1}a))ds
\end{equation}
Next we define a map $\phi_{A}:{\mathbb{E}}_{A}^{s}(r)\rightarrow{\mathbb{E}%
}_{A}^{u}$, $b\mapsto-\int_{0}^{\infty}I_{\Gamma_{2}}(-s)F(u(s,b))ds$ for
$b\in{\mathbb{E}}_{A}^{s}(r)$, where $r$ is the rational number defined in
(\ref{er}) and ${\mathbb{E}}_{A}^{s}(r)=\{b\in{\mathbb{E}}_{A}^{s}:|b|\leq
r/2\}$. We observe that the compact set ${\mathbb{E}}_{A}^{s}(r)$ is
computable since the closed set ${\mathbb{E}}_{A}^{s}$ is computable (proved
in Theorem \ref{t0}).
%(both $\{ b\in {\mathbb E}^s_A: |b|<r/2\}$ and $\{ b\in {\mathbb
%E}^s_A: |b|>r/2\}$ are r.e. open in ${\mathbb E}^s_A$, which imply
%that both ${\mathbb E}^s_A(r)$ and $\{ b\in {\mathbb E}^s_A: |b|\leq
%r/2\}$ are r.e. in ${\mathbb E}^s_A$; thus the compact set ${\mathbb
%E}^s_A(r)$ is computable in ${\mathbb E}^s_A$).
Obviously the map $\phi_{A}$ is computable. By Theorem 6.2.4 \cite{Wei00} the
compact set $\phi_{A}[{\mathbb{E}}_{A}^{s}(r)]$ is computable.

Now we are ready to define $S$:
\[
S=\{b+\phi_{A}(b):b\in{\mathbb{E}}_{A}^{s}(r)\}
\]
Then $S$ is a manifold of dimension $k$. For any $a\in S$, $a=b+\phi_{A}(b)$
for some $b\in\mathbb{E}_{A}^{s}(r)$. Since ${\mathbb{R}}^{n}={\mathbb{E}}%
_{A}^{s}\bigoplus{\mathbb{E}}_{A}^{u}$, $b\in{\mathbb{E}}_{A}^{s}$ and
$\phi_{A}(b)\in{\mathbb{E}}_{A}^{u}$, it follows that
\begin{equation}
\label{e-S}b=P_{1}a \quad\mbox{and} \quad\phi_{A}(b)=\phi_{A}(P_{1}a)=P_{2}a
\end{equation}
Combining (\ref{Eq:P1_a}), (\ref{e-P2}), and (\ref{e-S}) we obtain that for
any $a\in S$, $u(0,a)=P_{1}u(0,a)+P_{2}u(0,a)=P_{1}a+\phi_{A}(P_{1}%
a)=P_{1}a+P_{2}a=a$. Thus for any $a\in S$, $u(0,a)=a$; that is, $a$ is the
initial condition of $u(t,a)$ at $t=0$. \newline
%We also observe that
%for any solution $v(t,a)$ to the integral equation, if $|v(0,a)|\leq
%2^{-d(m_0)}/2K$ and $v(t, a)$ is constructed by the successive
%approximations (\ref{e4}), then $v(0,a)\in S$.

\noindent\textbf{Claim 3.} $S$ is a computable closed subset of ${\mathbb{R}%
}^{n}$. \newline

\noindent\textbf{Proof.} Since ${\mathbb{E}}_{A}^{s}(r)$ is computable, there
is a computable sequence $\{b_{j}\}\subseteq{\mathbb{E}}_{A}^{s}(r)$ that is
effectively dense in ${\mathbb{E}}_{A}^{s}(r)$; that is, there is a computable
function $\psi:{\mathbb{N}}\rightarrow{\mathbb{N}}$ such that ${\mathbb{E}%
}_{A}^{s}(r)\subseteq\bigcup_{j=1}^{\psi(k)}B(b_{j},2^{-k})$ for all
$k\in{\mathbb{N}}$ (\emph{c.f.} \cite{Zho96}). The following estimate shows
that the computable sequence $\{b_{j}+\phi_{A}(b_{j})\}$ is effectively dense
in $S$, thus the closed manifold $S$ is computable. For any $b+\phi
_{A}(b),\tilde{b}+\phi_{A}(\tilde{b})\in S$, we have
\begin{align*}
&  |b+\phi_{A}(b)-(\tilde{b}+\phi_{A}(\tilde{b}))|\\
&  \leq|b-\tilde{b}|+|\phi_{A}(b)-\phi_{A}(\tilde{b})|\\
&  =|b-\tilde{b}|+\left\vert -\int_{0}^{\infty}I_{\Gamma_{2}}%
(-s)F(u(s,b))ds+\int_{0}^{\infty}I_{\Gamma_{2}}(-s)F(u(s,\tilde{b}%
))ds\right\vert \\
&  \leq|b-\tilde{b}|+\int_{0}^{\infty}||I_{\Gamma_{2}}(-s)||\cdot
|F(u(s,b))-F(u(s,\tilde{b}))|ds\\
&  \leq|b-\tilde{b}|+2^{-m_{0}}|u(s,b)-u(s,\tilde{b})|\int_{0}^{\infty
}Ke^{-\sigma s}ds\\
&  =|b-\tilde{b}|+\frac{K}{\sigma2^{m_{0}}}|u(s,b)-u(s,\tilde{b})|\\
&  =|b-\tilde{b}|+\frac{K}{\sigma2^{m_{0}}}\cdot3K|b-\tilde{b}|
\end{align*}
The estimates (\ref{eI}), (\ref{ine-0}), (\ref{er}), and (\ref{ine-3}) are
used in the above calculation. The proof of claim 3 is complete.

For every $a\in S$, set $x(t, a)=u(t,a)$; then by Claims 1, 2, and 3, $x(t,a)$
is the solution to the differential equation (\ref{e2}) with initial value $a$
at $t=0$ and $x(t, a)$ satisfies the conditions $(i)$ and $(ii)$ of Theorem
\ref{t1}.

Finally we show that if $x(t,x_{0})$ is a solution to the differential
equation (\ref{e2}) satisfying $0<|x_{0}|<2^{-d(m_{0})}/4K^{2}$ but
$x_{0}\not \in S$, then there exists $t^{\prime}>0$ such that $|x(t^{\prime
},x_{0})|> 2^{-d(m_{0})}$. This proves the last part of the theorem if we set
$\eta=2^{-d(m_{0})}$ and $D=\{ x\in\mathbb{R}^{n}: |x|<2^{-d(m_{0})}/4K^{2}%
\}$. Indeed, if otherwise $|x(t,x_{0})|\leq2^{-d(m_{0})}$ for all $t\geq0$. We
show in the following that this condition implies $x_{0}\in S$, which is a contradiction.

Since $x(t,x_{0})$ is the solution to $\dot{x}=Ax+F(x)$ with the initial value
$x_{0}$, it follows that (\emph{c.f.} Theorem 4.8.2 \cite{Rob95})
\[
x(t,x_{0})=e^{At}x_{0}+\int_{0}^{t}e^{(t-s)A}F(x(s,x_{0}))ds
\]
which, using (\ref{e-lin}), can be rewritten as
\begin{align}
\label{e-rep} &  x(t,x_{0})\\
&  = I_{\Gamma_{1}}(t)x_{0}+I_{\Gamma_{2}}(t)x_{0}+\int_{0}^{t}I_{\Gamma_{1}%
}(t-s)F(x(s,x_{0}))ds+\int_{0}^{t}I_{\Gamma_{2}}(t-s)F(x(s,x_{0}%
))ds\nonumber\\
&  = I_{\Gamma_{1}}(t)x_{0}+\int_{0}^{t}I_{\Gamma_{1}}(t-s)F(x(s,x_{0}%
))ds-\int_{t}^{\infty}I_{\Gamma_{2}}(t-s)F(x(s,x_{0}))ds+I_{\Gamma_{2}%
}(t)b\nonumber
\end{align}
where $b=x_{0}+\int_{0}^{\infty}I_{\Gamma_{2}}(-s)F(x(s,x_{0}))ds$
(\emph{c.f.} \S 1.5.3 \cite{Kat95}). Note that $b$ is well defined since we
assume that $|x(t,x_{0})|\leq2^{-d(m_{0})}$ for all $t\geq0$, then, from
(\ref{ine-0}), $|F(x(t,x_{0}))|$ is bounded for all $t\geq0$; consequently,
by\ (\ref{eI}), the integral $\int_{0}^{\infty}I_{\Gamma_{2}}(-s)F(x(s,x_{0}%
))ds$ converges. We also note that the first three terms in the above
representation for $x(t,x_{0})$ are bounded. Moreover, $I_{\Gamma_{2}%
}(t)b=I_{\Gamma_{2}}(t)P_{1}b+I_{\Gamma_{2}}(t)P_{2}b=I_{\Gamma_{2}}(t)P_{2}b$
since $I_{\Gamma_{2}}(t)P_{1}b=0$ by (\ref{e6}). We claim that if
$P_{2}b\not \equiv 0$, then $I_{\Gamma_{2}}(t)b$ is unbounded as
$t\rightarrow\infty$. We make use of the residue formula to prove the claim.
Recall that $\Gamma_{2}$ is a closed curve in the right-hand side of the
complex plane with counterclockwise orientation that contains in its interior
$\mu_{j}$, $k+1\leq j\leq n$, where $\mu_{j}$ are the eigenvalues of $A$ with
$Re(\mu_{j})>0$, which are the exact singularities of $R(\xi)=(A-\xi
I_{n})^{-1}$ in the right complex plane. Then by the residue formula,
\begin{equation}
\label{e-res}\frac{1}{2\pi i}\int_{\Gamma_{2}}e^{t\xi}R(\xi)d\xi=\sum
_{l=k+1}^{n}e^{\mu_{j}t}\mbox{res}_{\mu_{j}}R
\end{equation}
where $\mbox{res}_{\mu_{j}}R$ is the residue of $R$ at $\mu_{j}$. Since
$Re(\mu_{j})>0$, if $P_{2}b\not \equiv 0$, then
\begin{align*}
&  I_{\Gamma_{2}}(t)b = I_{\Gamma_{2}}(t)P_{2}b\\
&  =-\frac{1}{2\pi i}\int_{\Gamma_{2}}e^{t\xi}(A-\xi I_{n})^{-1}P_{2}bd\xi\\
&  =-\sum_{l=k+1}^{n}e^{\mu_{j}t}res_{\mu_{j}}RP_{2}b
\end{align*}
is unbounded as $t\to\infty$. This is however impossible because the first
three terms in (\ref{e-rep}) are bounded and we have assumed that
$|x(t,x_{0})|\leq2^{-d(m_{0})}$ for all $t\geq0$. Therefore, $P_{2}b\equiv0$;
consequently, $I_{\Gamma_{2}}(t)b=I_{\Gamma_{2}}(t)P_{2}b=0$. This last
equation together with (\ref{e-rep}) shows that $x(t,x_{0})$ satisfies the
integral equation (\ref{e3}). Now let $x^{\prime}(t, x_{0})$ be the solution
to the integral equation (\ref{e3}) with parameter $x_{0}$ and constructed by
the successive approximations (\ref{e4}). Then
\[
x^{\prime}(t, x_{0})=I_{\Gamma_{1}}(t)x_{0}+\int_{0}^{t}I_{\Gamma_{1}%
}(t-s)F(x^{\prime}(s, x_{0}))ds-\int_{t}^{\infty}I_{\Gamma_{2}}%
(t-s)F(x^{\prime}(s, P_{1}x_{0}))ds
\]
By the uniqueness of the solution, $x(t, x_{0})=x^{\prime}(t, x_{0})$; in
particular, $x_{0}=x(0, x_{0})=x^{\prime}(0, x_{0})=P_{1}x_{0}-\int
_{0}^{\infty}I_{\Gamma_{2}}(-s)F(x^{\prime}(s, P_{1}x_{0}))ds$. Since
$P_{1}x_{0}\in\mathbb{E}^{s}_{A}$ and $\norm{P_1x_0}\leq
\norm{P_1}\norm{x_0}\leq K\cdot2^{d(m_{0})}/4K^{2}=r/2$ (recall that
$\norm{P_1}\leq K$ by (\ref{eI})), it follows that $P_{1}x_{0}\in
\mathbb{E}^{s}_{A}(r)$, which further implies that $x_{0}=P_{1}x_{0}-\int
_{0}^{\infty}I_{\Gamma_{2}}(-s)F(x^{\prime}(s, P_{1}x_{0}))ds=P_{1}x_{0}%
+\phi_{A}(P_{1}x_{0})\in S$. This contradicts the fact that $x_{0}$ is not on
$S$. The proof is complete.
\end{proof}

\begin{theorem}
Let $f:\mathbb{R}^{n}\rightarrow\mathbb{R}^{n}$ be a $C^{1}$-computable
function. Assume that the origin $0$ is a hyperbolic equilibrium point of
(\ref{e0}) such that $Df(0)$ has $k$ eigenvalues with negative real part and
$n-k$ eigenvalues with positive real part (counting multiplicity), $0\leq
k<n$. Let $x(t,x_{0})$ denote the solution to (\ref{e0}) with the initial
value $x_{0}$ at $t=0$. Then there is a (Turing) algorithm that computes a
$(n-k)$-dimensional manifold $U\subset{\mathbb{R}}^{n}$ containing $0$ such that

\begin{enumerate}
\item[(i)] For all $x_{0}\in U$, $\lim_{t\rightarrow-\infty}x(t,x_{0})=0$;

\item[(ii)] There are three positive rational numbers $\gamma$, $\epsilon$,
and $\delta$ such that $|x(t,x_{0})|\leq\gamma2^{\epsilon t}$ for all $t\leq0$
whenever $x_{0}\in U$ and $|x_{0}|\leq\delta$.
\end{enumerate}

Moreover, if $k>0$, then a rational number $\eta$ and a ball $D$ can be
computed from $f$ such that for any solution $x(t,x_{0})$ to the equation
(\ref{e0}) with $x_{0}\in D\setminus U$, $\{ x(t,x_{0}): \ t\leq
0\}\not \subset B(0, \eta)$ no matter how close the initial value $x_{0}$ is
to the origin.
\end{theorem}

\begin{proof}
The unstable manifold $U$ can be computed by the same procedure as the
construction of $S$ by considering the equation
\[
\dot{x}=-Ax-F(x(t))
\]

\end{proof}

The proof can be easily extended to show that the map: ${\mathfrak{F}}%
_{H}\rightarrow{\mathfrak{K}}\times{\mathfrak{K}}$, $f\mapsto(S_{f},U_{f})$,
is computable, where ${\mathfrak{F}}_{H}$ is the set of all $C^{1}$ functions
having the origin as a hyperbolic equilibrium point, ${\mathfrak{K}}$ is the
set of all compact subsets of ${\mathbb{R}}^{n}$, and $S_{f}$ and $U_{f}$ are
some local stable and unstable manifolds of $f$ at the origin respectively.

%%%%%%%%%%%%%%%%%%%%%%%%%%%%%%%%%%%%%%%%%%%%%%%%%%%%%%%%%%%%%%%%%%%%%%%%%%%%%%%%%%%%%%%%%%%%%%%%
%Section 5 - Computable Smale horseshoe
%%%%%%%%%%%%%%%%%%%%%%%%%%%%%%%%%%%%%%%%%%%%%%%%%%%%%%%%%%%%%%%%%%%%%%%%%%%%%%%%%%%%%%%%%%%%%%%%

\section{Computability and non-computability of global stable/unstable
manifolds\label{Sec:NotComputable}}

Although the stable and unstable manifolds can be computed locally as shown in
the previous section, globally they may not be
necessarily computable.

In \cite{Zho09} it is shown that there exists a $C^{\infty}$ and
polynomial-time computable function $f: \mathbb{R}^{2}\to\mathbb{R}^{2}$ such
that the equation $\dot{x}=f(x)$ has a sink at the origin \textbf{0} and the
basin of attraction (also called the domain of attraction), $B_{f}%
(\mathbf{0})$, of $f$ at \textbf{0} is a non-computable open subset of
$\mathbb{R}^{2}$. Since a sink is a hyperbolic equilibrium point (with all
eigenvalues of the gradient matrix $Df(\mathbf{0})$ having negative real
parts) and the basin of attraction at a sink is exactly the global stable
manifold at this equilibrium point, we conclude immediately that the global
stable manifold is not necessarily computable.

One the other hand, it is also shown in \cite{Zho09} that, although not
necessarily computable, the basin of attraction $B_{f}(x_{0})$ is r.e. open
for a computable sink $x_{0}$ of a computable system $f$. In other words, the
open subset $B_{f}(x_{0})$ can be plotted from inside by a computable sequence
$\{B_{n}\}$ of rational open balls, $B_{f}(x_{0})=\bigcup B_{n}$, but one may
not know the rate at which $B_{f}(x_{0})$ is being filled up by $B_{n}$'s if
$B_{f}(x_{0})$ is non-computable. When an equilibrium point $x_{0}$ of the
system $\dot{x}=f(x)$ is a saddle (not a sink nor a source), then the global
stable (unstable) manifold $W_{f}^{s}(x_{0})$ ($W_{f}^{u}(x_{0})$) is in
general an $F_{\sigma}$-set of $\mathbb{R}^{n}$. As in the special case of the
basin of attraction, $W_{f}^{s}(x_{0})$ can also be plotted from inside by a
computable process. To make the concept precise, we introduce the following
definition. Let $\mathcal{F}(\mathbb{R}^{n})$ denote the set of all
$F_{\sigma}$-subsets of $\mathbb{R}^{n}$.

\begin{definition}
A function $f:X\rightarrow\mathcal{F}(\mathbb{R}^{n})$ is called
semi-computable if there is a Type-2 machine such that on any $\rho$-name of
$x\in X$, the machine computes as output a sequence $\{a_{j,k}\}$, $a_{j,k}%
\in\mathbb{Q}^{n}$, such that
\[
f(x)=\bigcup_{j=0}^{\infty}\overline{\{ a_{j,k} : k\in\mathbb{N}\}}
\]
where $\overline{A}$ denotes the closure of the set $A$.
\end{definition}

We call this function semi-computable because we can plot a dense subset of
the set $f(x)$ by a computable process, but we cannot tell in a finite amount
of time what the ``density" is.

\begin{theorem}
\label{t-5-1} The map $\psi: C^{1}(\mathbb{R}^{n})\times\mathbb{R}^{n}%
\to\mathcal{F}(\mathbb{R}^{n})$, $(f, x_{0})\to W^{s}_{f}(x_{0})$, is
semi-computable, where $(f, x_{0})\in\mbox{dom}(\psi)$ if $x_{0}$ is a
hyperbolic equilibrium point of $f$, and $W^{s}_{f}(x_{0})$ is the global
stable manifold of $f$ at $x_{0}$.
\end{theorem}

\begin{proof}
From Theorem \ref{t1} one can compute from $f$ and $x_{0}$ a compact subset
$S$ of $\mathbb{R}^{n}$, which is a local stable manifold of $f$ at $x_{0}$.
Note that the global stable manifold of $f$ at $x_{0}$ is the union of the
backward flows of $S$, $\emph{i.e.}$,
\[
W_{f}^{s}(x_{0})=\bigcup_{j=0}^{\infty}\phi_{-j}(S)
\]
where $\phi_{t}(a)$ is the flow induced by the equation $\dot{x}=f(x)$ at time
$t$ with the initial data $x(0)=a$. Since the sequence $\{\psi_{-j}(a)\}$ is
computable from $f$ and $a$ \cite{GZB07}, it follows that the sequence
$\{\phi_{-j}(S)\}$ of compact sets are computable from $f$ and $x_{0}$
(Theorem 6.2.4 \cite{Wei00}). In particular, a sequence $\{a_{j,k}%
\}_{j,k\in\mathbb{N}}\subset\mathbb{R}^{n}$ can be computed such that
$W_{f}^{s}(x_{0})=\bigcup_{j=0}^{\infty}\overline{\{a_{j,k}:k\in\mathbb{N}\}}$.
\end{proof}

The function $f$ in the counterexample mentioned at the beginning of this
section is $C^{\infty}$ but not analytic. It is an open problem whether or not
the stable/unstable manifold(s) of a computable analytic hyperbolic system
$\dot{x}=f(x)$ (\emph{i.e.} $f$ is computable and analytic) is computable. In
the following we present a negative answer to a weaker version of the open
problem; we show that it is impossible to compute uniformly the closure of the
global stable/unstable manifold from $(f, x_{0})$, where $f$ is analytic and
$x_{0}$ is a hyperbolic equilibrium point of $f$. Let's denote by
$\omega(\mathbb{R}^{n})$ the set of real analytic functions $f:\mathbb{R}%
^{n}\to\mathbb{R}^{n}$ and $\mathcal{A}(\mathbb{R}^{n})$ the set of all closed
subsets of $\mathbb{R}^{n}$. With Wijsman topology $\mathcal{W}$ on
$\mathcal{A}(\mathbb{R}^{n})$, $(\mathcal{A}(\mathbb{R}^{n}), \mathcal{W})$ is
a topological space that is separable and metrizable with a complete metric.
It can be shown that the Wijsman topology on $\mathcal{A}(\mathbb{R}^{n})$ is
the same as the topology induced by the $\rho$-names giving rise to the notion
of computable close subsets of $\mathbb{R}^{n}$ as defined in Def.
\ref{def_r_e_open} (\cite{BW99}). Thus if a map $F:X\to\mathcal{A}%
(\mathbb{R}^{n})$ is computable (with respect to above $\rho$-names for
elements of $\mathcal{A}(\mathbb{R}^{n})$), then $F$ is continuous (with
respect to Wijsman topology on $\mathcal{A}(\mathbb{R}^{n})$) (Corollary
3.2.12 of \cite{Wei00}).

\begin{theorem}
The map $\psi:\omega(\mathbb{R}^{n})\times\mathbb{R}^{n}\rightarrow
\mathcal{A}(\mathbb{R}^{n})$, $(f,x_{0})\rightarrow\overline{W_{f}^{u}(x_{0}%
)}$ (the closure of $W_{f}^{u}(x_{0})$), is not computable, where
$(f,x_{0})\in\mbox{dom}(\psi)$ if $x_{0}$ is a hyperbolic equilibrium point of
$f$, and $W_{f}^{u}(x_{0})$ is the global unstable manifold of $f$ at $x_{0}$.
\end{theorem}

\begin{proof}
Consider the following system $\dot{x}=f_{\mu}(x)$ taken from \cite{HSD04},%
\begin{align}
x^{\prime}  &  =x^{2}-1\label{Eq:big_system}\\
y^{\prime}  &  =-xy+\mu(x^{2}-1)\nonumber
\end{align}
The system (\ref{Eq:big_system}) has two equilibria: $z_{1}=(-1,0)$ and
$z_{2}=(1,0)$. $Df_{\mu}(z_{1})$ has eigenvalues $-2$ and $1$, associated to
the eigenvectors $(-1,-2\mu/3)$ and $(0,1)$, respectively, and $Df_{\mu}%
(z_{2})$ has eigenvalues $2$ and $-1$, associated to the eigenvectors
$(-1,-2\mu/3)$ and $(0,1)$, respectively. Therefore both points $z_{1}$ and
$z_{2}$ are saddles (the behavior of the system is sketched in
Fig.\ \ref{fig:dynamics}). From this information and the fact that any point
$(-1,y)$ or $(1,y)$ can only move along the $y$ axis, one concludes that the
line $x=-1$ is the unstable manifold of $z_{1}$ and the line $x=1$ gives the
stable manifold of $z_{2}$. Note that $z_{1},z_{2}$ and the above manifolds do
not depend on $\mu.$

Let us now study the unstable manifold of $z_{2}$. It is split by the stable
manifold $x=1$, and hence there is a \textquotedblleft right\textquotedblright%
\ as well as a \textquotedblleft left\textquotedblright\ portion of the
unstable manifold$.$ Let us focus our attention to the \textquotedblleft
left\textquotedblright\ portion. From the system (\ref{Eq:big_system}) one
concludes that any point $z$ near $z_{2}$, located to its left (i.e.\ $z<z_{2}%
$) will be pushed to the line $x=-1$ at a rate that is independent of the
$y$-coordinate of $z$.

When $\mu=0$, the eigenvector of $Df_{\mu}(z_{2})$ associated to the unstable
manifold is $(-1,0)$. Looking at (\ref{Eq:big_system}), one concludes that the
\textquotedblleft left\textquotedblright\ portion of the unstable manifold of
$z_{2}$ can only be the segment of line
\[
U_{l,0}=\{(x,0)\in\mathbb{R}^{2}|-1<x<1\}.
\]

Now let us analyze the case where $\mu<0$. Since $(-1,-2\mu/3)$ is the
eigenvector of $Df_{\mu}(z_{2})$ associated to the unstable manifold, as the
unstable manifold of $z_{2}$ moves to the left, its $y$-coordinate starts to
grow. The \textquotedblleft left\textquotedblright\ portion of the unstable
manifold is always above the line $y=0$ (if it could be $y=0$, the dynamics of
(\ref{Eq:big_system}) would push the trajectory upwards), and as soon as its
$x$-coordinate is less than 0 (this will eventually happen), the trajectory is
pushed upwards with $y$-coordinate converging to $+\infty$. Notice that the
closer $\mu$ is to $0$, the lesser the $y$-component of the trajectory grows,
and the closer to $z_{1}$ the unstable manifold will be.

If $U_{l,\mu}$ represents the left portion of the unstable manifold of $z_{2}$
and
\[
A=U_{l,0}\cup\{(-1,y)\in\mathbb{R}^{2}|0\leq y\}
\]
then one concludes that
\[
\lim_{\mu\rightarrow0^{-}}d(A,U_{l,\mu})=0
\]
where $d$ is the Hausdorff distance on $\mathbb{R}^{n}$.

Suppose that $\psi$ is computable. Then, in particular, the map $\chi
:\mathbb{R}\rightarrow\mathcal{A}(\mathbb{R}^{n})$ defined by $\chi
(\mu)=\overline{W_{f_{\mu}}^{u}(z_{2})}$, is also computable.
%(one can get the unstable manifold $W_{f_{\mu}}^{u}$ from the map
%$\psi$ by using the substitution $t\mapsto-t$. In practice, to get
%$\chi(\mu),$ we just need to feed $\psi$ with $(-f_{\mu},z_{2})$).
Since computable maps must be continuous (Corollary 3.2.12 of \cite{Wei00}),
this implies that $\chi$ should be a continuous map. But the point $\mu=0$ is
a point of discontinuity for $\chi$ since (without loss of generality, we
restrict ourselves to the semi-plane $x<1$)%
\[
\lim_{\mu\rightarrow0^{-}}\chi(\mu)=\lim_{\mu\rightarrow0^{-}}\overline
{U}_{l,\mu}=\overline{A}\neq\overline{U}_{l,0}=\chi(0).
\]
and hence this map cannot be computable, which implies that $\psi$ is not
computable either.
\end{proof}

\begin{center}
\begin{figure}[ptb]
\begin{center}
\includegraphics[width=13cm]{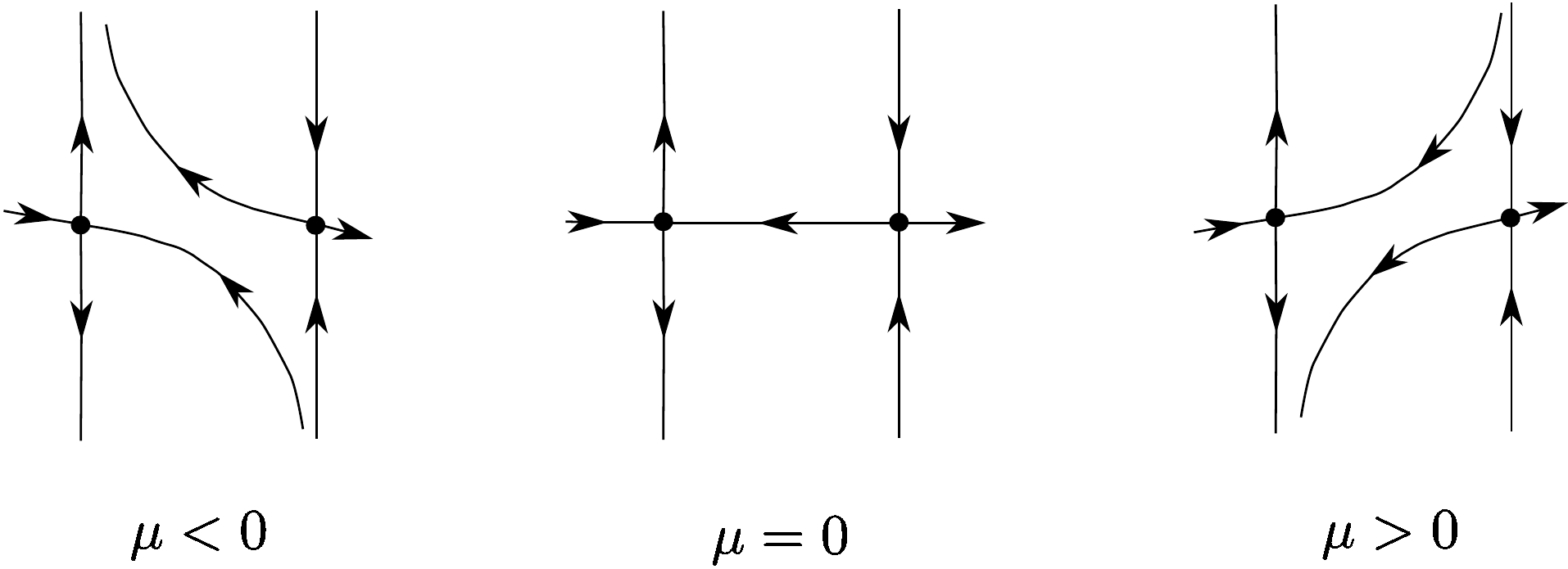}
\end{center}
\caption{Dynamics of the system \eqref{Eq:big_system}.}%
\label{fig:dynamics}%
\end{figure}
\end{center}

\section{The Smale horseshoe is computable}

In this section we show that Smale's horseshoe is computable.

\begin{theorem}
\label{thm_smale} The Smale Horseshoe $\Lambda$ is a computable (recursive)
closed subset in $I=[0,1]\times[0,1]$.
\end{theorem}

\begin{proof}
We show that $\Omega=I\setminus\Lambda$ is a computable open subset in $I$ by
making use of the fact \cite{Zho96}: An open subset $U\subseteq I$ is
computable if and only if there is a computable sequence of rational open
rectangles (having rational corner points) in $I$, $\{ J_{k}\}_{k=0}^{\infty}%
$, such that

\begin{itemize}
\item[(a)] $J_{k}\subset U$ for all $k\in{\mathbb{N}}$,

\item[(b)] the closure of $J_{k}$, $\bar{J}_{k}$, is contained in $U$ for all
$k\in{\mathbb{N}}$, and

\item[(c)] there is a recursive function $e:{\mathbb{N}}\to{\mathbb{N}}$ such
that the Hausdorff distance $d(I\setminus\cup_{k=0}^{e(n)}J_{k}, I\setminus
U)\leq2^{-n}$ for all $n\in{\mathbb{N}}$. \newline
\end{itemize}

Let $f:I\to{\mathbb{R}}^{2}$ be a map such that $\Lambda=\bigcap_{n=-\infty
}^{\infty}(f^{n}(I)\cap I)$ is the Smale horseshoe. Without loss of generality
assume that $f$ performs a linear vertical expansion by a factor of $\mu=4$
and a linear horizontal contraction by a factor of $\lambda=\frac{1}{4}$. For
each $n\in{\mathbb{N}}$, let $U_{n}=I\setminus\bigcap_{k=-n}^{n}(f^{k}(I)\cap
I)$. Then $I\setminus\Lambda=\bigcup_{n=0}^{\infty}U_{n}$. Moreover,

\begin{itemize}
\item[(1)] there exists a computable function $\alpha: {\mathbb{N}}%
\to{\mathbb{N}}$ such that $U_{n}$ is a union of $\alpha(n)$ rational open
rectangles in $I$,

\item[(2)] $d(\Lambda, I\setminus\bigcup_{k=0}^{n}U_{n})\leq(1/4)^{n}$.
\end{itemize}

Define $e:{\mathbb{N}}\to{\mathbb{N}}$, $e(n)=\sum_{k=0}^{n}\alpha(k)$,
$n\in{\mathbb{N}}$. Then it follows from the lemma that $I\setminus\Lambda$ is
indeed a computable open subset of $I$.
\end{proof}

\section{Conclusions}

We have shown that, locally, one can compute the stable and unstable manifolds
of some given hyperbolic equilibrium point, though globally these manifolds
are, in general, only semi-computable. It would be interesting to know if
these results are only valid to equilibrium points or can be extended e.g.\ to
hyperbolic periodic orbits. In \cite{GZ10} we provide an example which shows
that the global stable/unstable manifold cannot be computed for hyperbolic
periodic orbits. However, the question whether locally these manifolds can be
computed remains open.

\noindent\textbf{Appendix 1.} Proof of Claim 1. Assume that $u(t,a)$ is a
continuous solution to the integral equation (\ref{e3}). We show that it
satisfies the differential equation (\ref{e2}). To see this, we first
establish a relation between $I_{\Gamma_{j}}$, $j=1,2$, and the Jordan
canonical form of $A$:
\[
A=C\left(
\begin{array}
[c]{ll}%
P & 0\\
0 & Q
\end{array}
\right)  C^{-1}%
\]
where $C$ is an $n\times n$ invertible matrix, $P$ is a $k\times k$ matrix
with eigenvalues $\lambda_{j}$, $1\leq j\leq k$, and $Q$ is a $(n-k)\times
(n-k)$ matrix with eigenvalues $\mu_{j}$, $k+1\leq j\leq n$. By (\ref{e-lin}),
we have
\[
e^{At}=I_{\Gamma_{1}}(t)+I_{\Gamma_{2}}(t)
\]
On the other hand, it is straightforward to show that
\[
e^{At}=C\left(
\begin{array}
[c]{cc}%
e^{Pt} & 0\\
0 & 0
\end{array}
\right)  C^{-1}+C\left(
\begin{array}
[c]{ll}%
0 & 0\\
0 & e^{Qt}%
\end{array}
\right)  C^{-1}%
\]
Since the first $k$ columns of $C$ consists of a basis of the stable subspace
$P_{1}\mathbb{R}^{n}$ and the last $(n-k)$ columns of $C$ is a basis of the
unstable subspace $P_{2}\mathbb{R}^{n}$, it then follows from (\ref{e5}) and
(\ref{e6}) that
\[
I_{\Gamma_{2}}(t)C\left(
\begin{array}
[c]{cc}%
e^{Pt} & 0\\
0 & 0
\end{array}
\right)  C^{-1}=0,\quad\mbox{and}
\]%
\[
I_{\Gamma_{1}}(t)C\left(
\begin{array}
[c]{cc}%
0 & 0\\
0 & e^{Qt}%
\end{array}
\right)  C^{-1}=0
\]
Thus
\[
I_{\Gamma_{1}}(t)=C\left(
\begin{array}
[c]{cc}%
e^{Pt} & 0\\
0 & 0
\end{array}
\right)  C^{-1}%
\]
and
\[
I_{\Gamma_{2}}(t)=C\left(
\begin{array}
[c]{ll}%
0 & 0\\
0 & e^{Qt}%
\end{array}
\right)  C^{-1}%
\]
(recall that $\mathbb{R}^{n}=P_{1}\mathbb{R}^{n}\bigoplus P_{2}\mathbb{R}^{n}%
$). Consequently the integral equation (\ref{e3}) can be written in the
following form:
\begin{align}
u(t,a)  &  =C\left(
\begin{array}
[c]{cc}%
e^{Pt} & 0\\
0 & 0
\end{array}
\right)  C^{-1}a+\int_{0}^{t}C\left(
\begin{array}
[c]{cc}%
e^{P(t-s)} & 0\\
0 & 0
\end{array}
\right)  C^{-1}F(u(s,a))ds\label{e-jordan}\\
&  -\int_{t}^{\infty}C\left(
\begin{array}
[c]{ll}%
0 & 0\\
0 & e^{Q(t-s)}%
\end{array}
\right)  C^{-1}F(u(s,a))ds\nonumber
\end{align}
It is known that if $u(t,a)$ is a continuous solution to (\ref{e-jordan}),
then it is the solution to (\ref{e2}) (\emph{c.f.} \S 2.7 \cite{Per01}). The
proof is complete. \newline

\noindent\textbf{Appendix 2.} Proof of Claim 2. For $a,\tilde{a}\in B$ and
$t\geq0$, where $B=B(0,r)=\{x\in\mathbb{R}^{n}:|x|<r\}$, $r=2^{-d(m_{0})}/2K$
($d$ is the computable function defined in (\ref{ine-0}), and $m_{0}$ is a
positive integer such that $2^{-m_{0}}\leq\frac{\sigma}{4K}$, define
\begin{align*}
u^{(0)}(t,a)  &  =0\\
u^{(j)}(t,a)  &  =I_{\Gamma_{1}}(t)a+\int_{0}^{t}I_{\Gamma_{1}}%
(t-s)F(u^{(j-1)}(s,a))ds\\
&  -\int_{t}^{\infty}I_{\Gamma_{2}}(t-s)F(u^{(j-1)}(s,a))ds,\ \ j\geq1
\end{align*}
We show that the following three inequalities hold for all $j\in\mathbb{N}$:
\begin{align*}
|u^{(j)}(t,a)-u^{(j-1)}(t,a)|  &  \leq K|a|e^{-\alpha_{1}t}/2^{j-1}\\
|u^{(j)}(t,a)|  &  \leq2^{-d(m_{0})}e^{-\alpha_{1}t}\\
|u^{(j)}(t,a)-u^{(j)}(t,\tilde{a})|  &  \leq3K|a-\tilde{a}|
\end{align*}
We argue by induction on $j$. Since $u^{(0)}(t,a)=0$ for any $a$, by
(\ref{Eq:alphas}) and (\ref{eI}) we get $|u^{(1)}(t,a)|=||I_{\Gamma_{1}%
}(t)a||\leq Ke^{-(\alpha+\sigma)t}\cdot2^{-d(m_{0})}/2K<2^{-d(m_{0}%
)}e^{-\alpha_{1}t}$, and $|u^{(1)}(t,a)-u^{(1)}(t,\tilde{a})|=|I_{\Gamma_{1}%
}(t)(a-\tilde{a})|\leq Ke^{-(\alpha+\sigma)t}|a-\tilde{a}|\leq K|a-\tilde{a}%
|$, the three inequalities hold for $j=1$. The estimate (\ref{eI}) is used in
calculations here.

Assume that the three inequalities hold for all $k\leq j$. Then for $k=j+1$,
\begin{align*}
&  |u^{(j+1)}(t,a)-u^{(j)}(t,a)|\\
&  \leq\int_{0}^{t}||I_{\Gamma_{1}}(t-s)||\cdot|F(u^{(j)}(s,a))-F(u^{(j-1)}%
(s,a))|ds\\
&  +\int_{t}^{\infty}||I_{\Gamma_{2}}(t-s)||\cdot|F(u^{(j)}(s,a))-F(u^{(j-1)}%
(s,a))|ds\\
&  \leq\int_{0}^{t}||I_{\Gamma_{1}}(t-s)||\cdot2^{-m_{0}}|u^{(j)}%
(s,a)-u^{(j-1)}(s,a)|ds\\
&  +\int_{t}^{\infty}||I_{\Gamma_{2}}(t-s)||\cdot2^{-m_{0}}|u^{(j)}%
(s,a)-u^{(j-1)}(s,a)|ds\\
&  \leq2^{-m_{0}}\int_{0}^{t}Ke^{-(\alpha_{1}+\alpha_{2}+\sigma)(t-s)}%
\frac{K|a|e^{-\alpha_{1}s}}{2^{j-1}}ds+2^{-m_{0}}\int_{t}^{\infty}%
Ke^{\sigma(t-s)}\frac{K|a|e^{-\alpha_{1}s}}{2^{j-1}}ds\\
&  =2^{-m_{0}}\frac{K^{2}|a|}{2^{j-1}}e^{-(\alpha_{1}+\alpha_{2}+\sigma)t}%
\int_{0}^{t}e^{(\alpha_{2}+\sigma)s}ds+2^{-m_{0}}\frac{K^{2}|a|}{2^{j-1}%
}e^{\sigma t}\int_{t}^{\infty}e^{(-(\alpha_{1}+\sigma)s}ds\\
&  =2^{-m_{0}}\frac{K^{2}|a|}{2^{j-1}}e^{-(\alpha_{1}+\alpha_{2}+\sigma
)t}\frac{e^{(\alpha_{2}+\sigma)t}-1}{\alpha_{2}+\sigma}+2^{-m_{0}}\frac
{K^{2}|a|}{2^{j-1}}e^{\sigma t}\frac{0-e^{-(\alpha_{1}+\sigma)t}}{-(\alpha
_{1}+\sigma)}\\
&  \leq2^{-m_{0}}\frac{K^{2}|a|}{2^{j-1}}\frac{e^{-\alpha_{1}t}}{\alpha
_{2}+\sigma}+2^{-m_{0}}\frac{K^{2}|a|}{2^{j-1}}\frac{e^{-\alpha_{1}t}}%
{\alpha_{1}+\sigma}\\
&  <\frac{K|a|}{2^{j}}e^{-\alpha_{1}t}\quad(\mbox{recall that
$2^{-m_0}\leq \frac{\sigma}{4K}$})
\end{align*}
and furthermore,
\begin{align*}
&  |u^{(j+1)}(t,a)|\\
&  \leq|u^{(j)}(t,a)|+|u^{(j+1)}(t,a)-u^{(j)}(t,a)|\\
&  \leq|u^{(j)}(t,a)|+\frac{K|a|}{2^{j}}e^{-\alpha_{1}t}\\
&  \leq\sum_{k=1}^{j}\frac{K|a|}{2^{k}}e^{-\alpha_{1}t}\qquad
(\mbox{induction hypothesis on $u^{(k)}(t, a)$ for $k\leq j$})\\
&  \leq2K|a|e^{-\alpha_{1}t}\leq2Ke^{-\alpha_{1}t}\cdot2^{-d(m_{0}%
)}/2K=2^{-d(m_{0})}e^{-\alpha_{1}t}%
\end{align*}
Lastly we show that if $|u^{(k)}(t,a)-u^{(k)}(t,\tilde{a})|\leq3K|a-\tilde
{a}|$ holds for all $k\leq j$, then it holds for $j+1$.
\begin{align*}
&  |u^{(j+1)}(t,a)-u^{(j+1)}(t,\tilde{a})|\\
&  =\left\vert I_{\Gamma_{1}}(t)(a-\tilde{a})+\int_{0}^{t}I_{\Gamma_{1}%
}(t-s)\left(  F(u^{(j)}(s,a))-F(u^{(j)}(s,\tilde{a}))\right)  ds-\right. \\
&  -\left.  \int_{t}^{\infty}I_{\Gamma_{2}}(t-s)\left(  F(u^{(j)}%
(s,a))-F(u^{(j)}(s,\tilde{a}))\right)  ds\right\vert \\
&  \leq\left\vert I_{\Gamma_{1}}(t)(a-\tilde{a})\right\vert +\int_{0}%
^{t}||I_{\Gamma_{1}}(t-s)||\cdot\left\vert F(u^{(j)}(s,a))-F(u^{(j)}%
(s,\tilde{a}))\right\vert ds+\\
&  \int_{t}^{\infty}||I_{\Gamma_{2}}(t-s)||\cdot\left\vert F(u^{(j)}%
(s,a))-F(u^{(j)}(s,\tilde{a}))\right\vert ds\\
&  \leq K|a-\tilde{a}|+2^{-m_{0}}\left\vert u^{(j)}(s,a)-u^{(j)}(s,\tilde
{a})\right\vert \left(  \int_{0}^{t}Ke^{-(\alpha+\sigma)(t-s)}ds+\int
_{t}^{\infty}Ke^{\sigma(t-s)}ds\right) \\
&  \leq K|a-\tilde{a}|+2^{-m_{0}}\cdot3K|a-\tilde{a}|\cdot\left(  \frac
{K}{\alpha+\sigma}+\frac{K}{\sigma}\right) \\
&  =K|a-\tilde{a}|\left(  1+2^{-m_{0}}\frac{3K}{\alpha+\sigma}+2^{-m_{0}}%
\frac{3K}{\sigma}\right) \\
&  \leq3K|a-\tilde{a}|\quad
(\mbox{Recall that $2^{-m_0}\leq \frac{\sigma}{4K}$})
\end{align*}
The proof is complete.\smallskip

\textbf{Acknowledgments.}
D.\ Gra\c{c}a was
partially supported by \emph{Funda\c{c}\~{a}o para a Ci\^{e}ncia e a
Tecnologia} and EU FEDER POCTI/POCI via SQIG - Instituto de
Telecomunica\c{c}\~{o}es. N. Zhong was partially supported by the Charles
Phelps Taft Memorial Fund of the University of Cincinnati.

\bibliographystyle{alpha}

\end{document}